\documentclass[12pt]{amsart}
\usepackage[pagebackref=false,hyperindex=true,colorlinks=false, pdfauthor={Olivier Haution}, pdftitle={Duality and the topological filtration}, pdfkeywords={ {S}teenrod operations, {C}how groups, connective $K$-theory, Quillen spectral sequence, {A}dams operations}]{hyperref}

\usepackage{enumerate,amssymb,version}
\usepackage{geometry}
\usepackage[all]{xy}

\DeclareMathOperator{\Ch}{Ch}
\DeclareMathOperator{\CH}{CH}
\DeclareMathOperator{\id}{id}
\DeclareMathOperator{\Spec}{Spec}

\DeclareMathOperator{\im}{im}

\DeclareMathOperator{\Cht}{\widetilde{\Ch}}
\DeclareMathOperator{\CK}{CK}
\DeclareMathOperator{\Ck}{Ck}
\DeclareMathOperator{\coker}{coker}
\DeclareMathOperator{\rank}{rank}
\DeclareMathOperator{\Homm}{Hom}

\newcommand{\K}{K}
\newcommand{\Kg}[2]{K'_{#2}({#1})}
\newcommand{\Kh}[2]{K_{#2}({#1})}
\newcommand{\Khs}[3]{K_{#3}^{#2}({#1})}

\newcommand{\Var}{\mathsf{Var}}

\DeclareMathOperator{\gr}{gr}
\newcommand{\Kzl}{\K_0}
\newcommand{\Kzh}{\K_0}

\newcommand{\Zu}{\Z[1/k]}

\DeclareMathOperator{\Fi}{F}

\newcommand{\Figh}[3]{\Fi_{\gamma}^{#1} \K_{#3}({#2})}

\newcommand{\Gm}{\mathbb{G}_m}

\newcommand{\Tan}{T}
\newcommand{\PP}{\mathbb{P}}
\newcommand{\Aa}{\mathbb{A}^1}

\DeclareMathOperator{\Square}{Sq}
\newcommand{\Squ}{\Square_1}
\newcommand{\Squu}{\Square^1}

\newcommand{\colim}[1]{\mycolim_{#1}}
\DeclareMathOperator*{\mycolim}{colim}

\newcommand{\Ec}{\mathcal{E}}
\newcommand{\Fc}{\mathcal{F}}

\newcommand{\Cc}{\mathcal{C}}

\renewcommand{\Mc}{\mathcal{M}}

\newcommand{\Nn}{\mathcal{N}}

\newcommand{\by}{\beta^{\infty}}

\newcommand{\sik}{\tau_k}
\newcommand{\siu}{\tau_{-1}}

\newcommand{\Z}{\mathbb{Z}}

\newcommand{\point}{\mathsf{point}}
\newcommand{\Oc}{\mathcal{O}}
\newcommand{\Ab}{\mathsf{Ab}}

\newtheorem{theorem}{Theorem}[section]
\newtheorem{proposition}[theorem]{Proposition}
\newtheorem{lemma}[theorem]{Lemma}
\newtheorem{corollary}[theorem]{Corollary}

\theoremstyle{definition}
\newtheorem{remark}[theorem]{Remark}
\newtheorem{example}[theorem]{Example}
\newtheorem{definition}[theorem]{Definition}

\begin{document}
\begin{abstract}
We investigate some relations between the duality and the topological filtration in algebraic $K$-theory. As a result, we obtain a construction of the first Steenrod square for Chow groups modulo two of varieties over a field of arbitrary characteristic. This improves previously obtained results, in the sense that it is not anymore needed to mod out the image modulo two of torsion integral cycles. Along the way we construct a lifting of the first Steenrod square to algebraic connective $K$-theory with integral coefficients, and homological Adams operations in this theory. Finally we provide some applications to the Chow groups of quadrics.
\end{abstract}
\author{Olivier Haution}
\title{Duality and the topological filtration}
\email{olivier.haution@gmail.com}
\address{Mathematisches Institut, Ludwig-Maximilians-Universit\"at M\"unchen, Theresienstr.\ 39, D-80333 M\"unchen, Germany}


\date{\today}

\maketitle
\setcounter{tocdepth}{1}
\tableofcontents

\section{Introduction}
In this text we construct one of the Steenrod operations for Chow groups modulo two. These operations have proven to be efficient tools in the study of splitting properties of projective homogeneous varieties. Some striking examples related to the theory of quadratic forms can be found in the book~\cite{EKM}.\\ 

Let $\Ch$ denote the Chow group modulo two, considered as a functor from the category of varieties and projective morphisms to the category of graded abelian groups. Let $\Cht$ be its quotient by the image of torsion integral cycles. In \cite{firstsq}, we have constructed a natural transformation, for every integer $p$,
\[
\Squ \colon \Ch_p \to \Cht_{p-1}.
\]

Although this construction is sufficient for the applications that we have in mind, it is not the strongest possible form of this operation. As mentioned in \cite[Remark~5.1]{firstsq}, one expects this operation to lift to an operation
\[
\Squ \colon \Ch_p \to \Ch_{p-1}.
\]

This is known over fields of characteristic not two \cite{Bro-St-03,EKM,Lev-St-05,Vo-03}. The main purpose of this paper is to provide such a lifting over an arbitrary field. This is satisfying from a theoretical point of view, and realizing the first Steenrod square as an endomorphism of some group allows us to discuss the validity of the Adem relation $\Squ \circ \Squ=0$ in the end of the text.\\

Another motivation is that Steenrod operations constitute one of the rare tools which can provide useful informations about the torsion cycles. Therefore it seems to us interesting to dispose of a version of the first Steenrod square which can produce torsion cycles (modulo two). There is an epimorphism
\[
\varphi \colon \Ch_p \to \Z/2 \otimes \gr_p \Kzl',
\]
from the modulo two Chow group to the graded group associated with the topological filtration (modulo two). Its kernel is certainly a subtle object in general, and the first Steenrod square can actually provide some information about it. Its elements are always the image modulo two of torsion integral cycles \cite[Proposition~4.1, (vi)]{reduced}, but this condition is not sufficient (see for instance \eqref{Chxrho} and \eqref{Kxrho} in this text). We prove that the first Steenrod square descends to a map
\[
\Z/2 \otimes \gr_p \Kzl' \to \Z/2 \otimes \gr_{p-1} \Kzl'.
\]
Even over fields of characteristic zero, this statement seems new (and is indeed not true for all Steenrod squares, see Remark~\ref{rem:other}). An additional necessary condition for an element to vanish under $\varphi$ is therefore that its image under the first Steenrod square admits an integral lifting which is torsion. In the last section we describe an explicit situation where these two conditions are sufficient (and non-redundant). We also use the first Steenrod square to produce $2$-torsion elements in the first Chow group of a large class of quadrics.\\

We proceed by lifting the involution $\psi_{-1}$ induced by duality on $K$-theory to algebraic connective $K$-theory $\CK$, a ``deformation'' between $K$-theory and Chow groups introduced in \cite{Cai}. This allows us to define another operation $\siu$ on $\CK$, which happens to be an integral lifting of the first Steenrod square from Chow groups modulo two to algebraic connective $K$-theory. The existence of this operation seems to be a new result, at least over fields of positive characteristic (in characteristic zero, one can use algebraic cobordism, which, as a general fact, provides an efficient way to construct operations on oriented Borel-Moore functors). Rather than using Grothendieck duality, we derive the properties of $\psi_{-1}$ from the fact that it is an Adams operation. Indeed most of the time we consider the $k$-th Adams operation ($k \in \Z-\{0\}$); only in the very end do we take $k=-1$.

We also construct the $k$-th homological Adams operation for connective $K$-theory (over an arbitrary field); this has been done in characteristic zero in \cite{MalagonLopez}.\\

{\bf Acknowledgements.} Making the Adams operations act on Quillen spectral sequence is certainly very classical. A recent illustration can be found in \cite{Mer-BGQ}, whence we borrowed some of the arguments and references used in this text. The idea of constructing the first Steenrod square using the duality theory for schemes was inspired by \cite{Totaro-04}. I am very grateful to Nikita Karpenko, who drew my attention to connective $K$-theory and Adams operations while I was trying to construct Steenrod operations for Chow groups. Finally I would like to thank Baptiste Calm\`es and the anonymous referee, who made suggestions yielding to improvements in the exposition. The support of EPSRC Responsive Mode grant EP/G032556/1 is gratefully acknowledged.

\section{Basic facts}
\subsection{Varieties} \label{sect:var} We work over a fixed base field, whose spectrum shall be denoted by $\point$. A \emph{variety} is a quasi-projective separated scheme of finite type over this field. If $x$ is point of a variety $X$, we write $\kappa(x)$ for the residue field at $x$.

When $\Omega$ is a variety, we consider the category $\Var/\Omega$. Objects are varieties $X$ endowed with a morphism of varieties $X \to \Omega$, and morphisms are the proper morphisms of varieties which respect the structure of $\Omega$-schemes. When $\Omega=\point$, we write $\Var$ for $\Var/\Omega$.
 
\subsection{Regularity} A \emph{local complete intersection morphism} is a morphism $f$ which admits a factorization $p \circ i$ with $i$ a regular closed embedding, and $p$ a smooth morphism. Its \emph{virtual tangent bundle} $\Tan_f \in \Kzl(X)$ is $i^*[\Tan_p]-[N_i]$, where $\Tan_p$ is the tangent bundle of $p$, and $N_i$ the normal bundle of $i$. This element does not depend on the choice of the factorization $p \circ i$. 

A variety $X$ is \emph{regular} if the rings $\Oc_{X,x}$ are regular local rings, for all points $x$ of $X$. Since $X$ is quasi-projective, there is a smooth variety $M$, and a closed embedding $X \hookrightarrow M$. This embedding has to be regular \cite[Proposition~2, \S5, N$^{\circ}$3, p.65]{Bou-AC-10}, hence the structural morphism $X \to \point$ is a local complete intersection morphism. We write $\Tan_X$ for its virtual tangent bundle.

\subsection{$K$-groups} If $V$ is a scheme, we write $\Kh{V}{m}$ for the $m$-th $K$-group of locally-free coherent $\Oc_V$-modules. These groups admit pull-back along arbitrary morphisms of schemes, and Adams operations $\psi^k$, for $k\in \Z-\{0\}$, which are compatible with the pull-back maps.

If $R$ is a commutative ring, we also write $\Kh{R}{m}$ for $\Kh{\Spec(R)}{m}$.

\subsection{$\K'$-groups}
If $X$ is a noetherian scheme, we denote by $\Kg{X}{m}$ the $m$-th $K$-group of the category $\Mc(X)$ of coherent $\Oc_X$-modules. For a proper morphism of noetherian schemes $f \colon Y \to X$, with $Y$ supporting an ample line bundle, there is a push-forward $f_* \colon \Kg{Y}{m} \to \Kg{X}{m}$. If $f \colon Y \to X$ is a morphism of finite Tor-dimension between noetherian schemes, and if $X$ supports an ample line bundle, there is a pull-back $f^* \colon \Kg{X}{m} \to \Kg{Y}{m}$. See \cite[\S 7, 2.]{Qui-72} for more details.

When $f$ is a regular closed embedding of varieties, there is an alternative construction of $f^*$ recalled in \ref{sect:def} below.

\subsection{Notations} \label{sect:notations} If $l \colon F \to G$ is a natural transformation between two functors from a category $\mathcal{A}$ to a category $\mathcal{B}$, and $A$ an object of $\mathcal{A}$, we will freely write either  $l$ of $l^A$ for the map $F(A) \to G(A)$, depending on the situation.

We denote by $\Ab$ the category of abelian groups. If $M\in \Ab$, $A=\Z[1/k]$ or $A=\Z/n$ (with $k,n$ integers), and $x \in M$, we will usually write $x$ for the element $1 \otimes x \in A \otimes M$. If $N$ is another abelian group, and $f \colon M \to N$ a group homomorphism, we write $f$ for the map $\id \otimes f \colon A \otimes M \to A \otimes N$.

If $F$ is a field, $L/F$ a field extension, and $X$ a variety over $F$, then the variety $X \times_{\Spec(F)} \Spec(L)$ shall be denoted by $X_L$. Similarly, if $x \in \Kg{X}{m}$, the element $x_L \in \Kg{X_L}{m}$ will denote the image of $x$ under the pull-back along the flat morphism $X_L \to X$.

\subsection{Quillen spectral sequence} \cite[\S7, Theorem~5.4]{Qui-72}\label{sect:QSS} Let $X$ be a variety. For every integer $p$, we write $\Mc_p(X)$ for the full subcategory of $\Mc(X)$, with objects the coherent $\Oc_X$-modules supported in dimension at most $p$. The category $\Mc_p(X)$ is a full subcategory of $\Mc_{p+1}(X)$, and a combination of localization and d\'evissage yields an exact sequence 
\[
\cdots \to E^1_{p+1,q} \to \K_{p+q}(\Mc_p(X)) \to \K_{p+q}(\Mc_{p+1}(X)) \to E^1_{p+1,q-1} \to \cdots
\]
with ($X_{(p)}$ being the set of points of dimension $p$ in $X$)
\begin{equation}
\label{eq:e1}
E^1_{p,q}=\coprod_{x \in X_{(p)}} \Kh{\kappa(x)}{p+q}.
\end{equation}

Thus we have an exact couple which induces Quillen spectral sequence
\[
E^1_{p,q} \Rightarrow \Kg{X}{p+q},
\]
the filtration on the abutment being the topological filtration. 

\subsection{Connective $K$-theory} \cite{Cai} \label{par:connective} Let $X$ be a variety. The groups
\[
\CK_{p,q}(X)=\im(\K_{p+q}(\Mc_p(X)) \to \K_{p+q}(\Mc_{p+1}(X)) )
\]
are the \emph{connective $K$-theory groups} of $X$. They fit in a long exact sequence
\[
\cdots \to \CK_{p-1,q+1}(X) \xrightarrow{\beta} \CK_{p,q}(X) \to E^2_{p,q} \to \CK_{p-2,q+1}(X) \xrightarrow{\beta} \CK_{p-1,q}(X)\to \cdots
\]
and we have \cite[Proposition~5.26]{Sri-96}
\[
E^2_{p,-p} \simeq \CH_p(X).
\]

We consider $\CK_{p,q}$ as a functor $\Var \to \Ab$. When $Y \hookrightarrow X$ is a closed embedding, there is a localization sequence \cite[Theorem~5.1]{Cai}
\[
\cdots \to \CK_{p+1,q}(X-Y) \to \CK_{p,q}(Y) \to  \CK_{p,q}(X)\to  \CK_{p,q}(X-Y)\to \cdots
\]

\subsection{External product}
The external product for the groups $\Kg{-}{m} , \CK_{p,q}(-)$ and $\K_{p+q}(\Mc_p(-))$ will always be denoted by the symbol $\times$.

\subsection{Deformation to the normal cone} \cite[Chapter~5]{Ful-In-98} \label{sect:def}
Let $f \colon Y \hookrightarrow X$ be a closed embedding of varieties, and $N_f$ its normal cone \cite[B.6.1]{Ful-In-98}. Consider the blow-up $B$ of $X \times \Aa$ along $Y \times \{0\}$. The exceptional divisor is the projective cone $\PP(N_f \oplus 1)$ over $Y$. There is a closed embedding $\PP(N_f) \hookrightarrow \PP(N_f \oplus 1)$, with open complement $N_f$. Let $D_f$ be the open complement in $B$ of $\PP(N_f)$. Then $D_f$ contains $N_f$ as a locally principal divisor, with open complement isomorphic to $\Gm \times X$. This gives a commutative diagram
\[
\xymatrix{
N_f \ar[r] & D_f & \Gm \times X \ar[l]\\
Y \ar[r] \ar[u] & \Aa \times Y \ar[u] & \Gm \times Y \ar[l] \ar[u]_{\id \times f}
}
\]
the arrow on the left being the zero section of the cone $N_f$. 

Let $F$ be the base field, $R=F[t,t^{-1}]$, so that $\Gm=\Spec(R)$. We have a map 
\[
R^{\times} = GL_1(R) \to GL_{\infty}(R) \to GL_{\infty}(R)/[GL_{\infty}(R),GL_{\infty}(R)]=\Kh{R}{1}.
\]
Since $\Gm$ is regular, we have $\Kh{R}{1}=\Kh{\Gm}{1}=\Kg{\Gm}{1}$, hence the element $t \in R^{\times}$ gives an element $\{t\} \in \Kg{\Gm}{1}$. We also denote by $\{t\}$ its image under the map $\Kg{\Gm}{1} \to \CK_{1,0}(\Gm)$. Let $\delta^f \colon \Kg{\Gm \times X}{m+1} \to \Kg{N_f}{m}$ (resp.\ $\partial^f\colon\CK_{p+1,q}(\Gm \times X) \to \CK_{p,q}(N_f)$) be the connecting homomorphism in the localization sequence for $(N_f,D_f)$. The \emph{deformation homomorphisms} of $f$ are
\[
\begin{array}{rrlrrrl}
		  \sigma^f \colon &\Kg{X}{m} \to & \Kg{N_f}{m} &\quad \text{and} \quad &\sigma^f \colon &\CK_{p,q}(X)\to&\CK_{p,q}(N_f)\\
		  &x \mapsto &\delta^f(\{t\} \times x)& \quad& &x \mapsto &\partial^f(\{t\} \times x).
       		\end{array}
\]

When $f$ is a regular closed embedding of relative dimension $d$, the cone $n \colon N_f \to Y$ is a vector bundle, and the pull-back $n^*$ is an isomorphism \cite[\S 7, Proposition~4.1]{Qui-72} (resp.\ \cite[Theorem~5.3]{Cai}). One defines the \emph{Gysin morphism} $f^* \colon \Kg{X}{m} \to \Kg{Y}{m}$ (resp.\ $\CK_{p,q}(X) \to \CK_{p-d,q+d}(Y)$) by the formula 
\[
f^*=(n^*)^{-1} \circ \sigma^f.
\]
See \cite[Proposition~86]{Gil-K-05} and \cite[\S 6.3]{Cai} for more details. 

\subsection{The map $\by$}
\label{sect:by} Let $X$ be a variety. Using the map $\beta \colon \CK_{p,q}(X) \to \CK_{p+1,q-1}(X)$, we get a map
\[
\by=\colim{n\geq 0}\beta^n \colon \CK_{p,q}(X) \to \Kg{X}{p+q}=\colim{n\geq 0}\CK_{p+n,q-n}(X).
\]
The next two statements are immediate.
\begin{proposition}
\label{prop:properflataction}
The map $\by$ is compatible with proper push-forwards, flat pull-backs, external products, and action of $\Kzl$.
\end{proposition}

\begin{lemma}
\label{lemm:tby}
We have $\by\{t\}=\{t\}$ (see \ref{sect:def} for the definitions of $\{t\}$).
\end{lemma}

\begin{lemma}
\label{lemm:partialby}
Let $Y\hookrightarrow X$ be a closed embedding. We have a commutative diagram
\[ \xymatrix{
\CK_{p,q}(X-Y)\ar[rr]_{\by} \ar[d]_{\partial} && \Kg{X-Y}{p+q} \ar[d]_{\delta} \\ 
\CK_{p-1,q}(Y) \ar[rr]_{\by} && \Kg{Y}{p+q-1}
}\]
\end{lemma}
\begin{proof}
This follows from the construction of the connecting homomorphism $\partial$ in \cite[Theorem~5.1]{Cai} (see also \eqref{eq:deltabeta} below).
\end{proof}

\begin{lemma}
\label{lemm:deformby}
Let $f \colon Y\hookrightarrow X$ be a closed embedding with normal cone $N_f$. We have a commutative diagram
\[ \xymatrix{
\CK_{p,q}(X)\ar[rr]_{\by} \ar[d]_{\sigma^f} && \Kg{X}{p+q} \ar[d]_{\sigma^f} \\ 
\CK_{p,q}(N_f) \ar[rr]_{\by} && \Kg{N_f}{p+q}
}\]
\end{lemma}
\begin{proof}
Let $x \in \CK_{p,q}(X)$. Then, using the notations of \ref{sect:def}, Proposition~\ref{prop:properflataction}, Lemmas~\ref{lemm:tby} and \ref{lemm:partialby}, we have in $\Kg{N_f}{p+q}$
\begin{align*} 
\by \circ \sigma^f(x)&=\by \circ \partial^f (\{t\} \times x)\\
&=\delta^f \circ \by(\{t\} \times x)\\
&=\delta^f(\{t\} \times \by(x))\\
&=\sigma^f \circ \by(x).\qedhere
\end{align*}
\end{proof}

\begin{proposition}
The map $\by$ is compatible with pull-backs along local complete intersection morphisms.
\end{proposition}
\begin{proof}
This follows from Lemma~\ref{lemm:deformby} and Proposition~\ref{prop:properflataction}, since such pull-backs are constructed using deformation homomorphisms and flat pull-backs (see \ref{sect:def}).
\end{proof}

\section{Homological Adams operations}
The letter $m$ will denote an integer $\geq 0$, and $k$ will be an element of $\Z - \{0\}$.\\

We begin by recalling the definition of Bott's class $\theta^k$ which, in the terminology of \cite{Pan-Ri-02}, plays the role of the inverse Todd genus of the Adams operation $\psi^k$. When $k \in \Z -\{0\}$, consider the Laurent polynomial in the variable $u$
\[
t^k(u)=\frac{1-u^k}{1-u}=u^0 + \cdots + u^{k-1}.
\]
There exists a Laurent polynomial $r$ such that
\[
t^k(u)=k + (1-u) \cdot r(u).
\]
Thus if $v$ is an element in a commutative ring $R$ such that $1-v$ is nilpotent, then 
\[
t^k(v) \in (\Zu \otimes R)^{\times}.
\]

Let $\mathcal{S}$ be the full subcategory of schemes with objects the noetherian schemes possessing an ample line bundle. If $X \in \mathcal{S}$, and $L$ is a line bundle over $X$, we write $c_1(L)$ for the element $1 -[L^\vee] \in \Kzl(X)$. It is nilpotent by \cite[Corollary~3.10]{FL-Ri-85}. We define a morphism of presheaves of groups on $\mathcal{S}$
\[
\theta^k \colon \Kzl \to (\Zu \otimes \Kzl)^{\times},
\]
by the requirement that, for any line bundle $L$,
\[
\theta^k(L)=t^k[L^\vee].
\]
We have
\[
c_1(L) \cdot \theta^k(L)=c_1(L^{\otimes k}).
\]

Let $X \in \mathcal{S}$, and $y \in \Kzl(X)$. Using the splitting principle, we see that
\begin{equation}
\label{eq:thetakn}
\rank ( \theta^k(y) )=k^{\rank(y)}.
\end{equation}

If $M$ is a $\Kzl(X)$-module, then we denote by $\theta^k(y)$ the automorphism of $\Zu \otimes M$ given by the action of $\theta^k(y) \in (\Zu \otimes \Kzl(X))^{\times}$.\\

Let $X \hookrightarrow M$ be a closed embedding of varieties, with $M$ regular. There is a notion of \emph{$\K$-groups of $M$ with supports in $X$} \cite[\S4.2]{Sou-Op-85}, denoted $\Khs{M}{X}{m}$, and fitting in long localization sequences
\[
\cdots \to \Kh{M-X}{m+1} \to \Khs{M}{X}{m} \to \Kh{M}{m} \to \Kh{M-X}{m} \to \cdots
\]

If $i \colon M \hookrightarrow N$ is a closed embedding, with $N$ also regular, then there is an isomorphism $i_* \colon \Khs{M}{X}{m} \to \Khs{N}{X}{m}$. Indeed, the group $\Khs{M}{X}{m}$ is canonically isomorphic to $\Kg{X}{m}$, but the datum of the closed embedding $X \hookrightarrow M$ can be used to define \emph{Adams operations with supports} \cite[\S4.3]{Sou-Op-85}
\[
\psi^k_M \colon \Khs{M}{X}{m} \to \Khs{M}{X}{m}.
\]

If $X$ is a given variety, then one can always choose a smooth variety $M$ containing $X$ as a closed subvariety. Then the endomorphism $\theta^k(-\Tan_M) \circ \psi^k_M$ of $\Zu \otimes\Khs{M}{X}{m}\simeq \Zu \otimes\Kg{X}{m}$ does not depend on the choice of $M$ \cite[Th\'eor\`eme 7]{Sou-Op-85}. One can thus define \emph{homological Adams operations} $\psi_k$ on $\Zu \otimes \Kg{X}{m}$; we summarize some of their main properties in the next proposition.

\begin{proposition}[{\cite[Th\'eor\`eme 7]{Sou-Op-85}}]
\label{prop:adams}
For every variety $X$, the $k$-th homological Adams operation 
\[
\psi_k \colon \Zu \otimes \Kg{X}{m}\to \Zu \otimes \Kg{X}{m}
\]
has the following properties.
\begin{enumerate}[i)]
\item \label{adams:proper} If $f$ is a proper morphism of varieties, then $f_* \circ \psi_k=\psi_k \circ f_*$.

\item \label{adams:open} If $u$ is an open embedding of varieties, then $u^* \circ \psi_k=\psi_k \circ u^*$.
    
\item \label{adams:compos} We have, as maps $\Z[1/(kk')] \otimes \Kg{X}{m} \to
\Z[1/(kk')] \otimes \Kg{X}{m}$,
\[
\psi_k \circ \psi_{k'}=\psi_{kk'}.
\]
\item \label{item:extprod} For all $x \in \Kg{X}{m}$ and $y \in \Kg{Y}{n}$, we have in $\Kg{X \times Y}{m+n}$
\[
\psi_k(x \times y)=\psi_k(x) \times \psi_k(y).
\]
\end{enumerate}
\end{proposition}

The next statement was proved in \cite[Th\'eor\`eme~7, vi)]{Sou-Op-85}, for $X$ smooth.
\begin{proposition}
\label{prop:regular}
Let $X$ be a regular variety with virtual tangent bundle $\Tan_X \in \Kzh(X)$. Under the identification $\Kg{X}{m} \simeq \Kh{X}{m}$, we have
\[
\psi_k=\theta^k(-\Tan_X)\circ \psi^k.
\]
\end{proposition}
\begin{proof}
Choose a closed embedding $i \colon X \hookrightarrow M$, with $M$ smooth. Then $i$ is a regular closed embedding, let $N$ be its normal bundle. We have $\Tan_X +[N]=i^*\Tan_M$ in $\Kzh(X)$. By the Riemann-Roch theorem \cite[Th\'eor\`eme 3]{Sou-Op-85}, we have in $\Zu \otimes \Khs{M}{X}{m}$
\[
\theta^k(-\Tan_M) \circ \psi^k_M \circ i_*=\theta^k(-\Tan_M) \circ i_* \circ \theta^k(N)\circ \psi^k_X=i_* \circ \theta^k(-\Tan_X) \circ \psi^k_X.
\]
Here $i_*$ is the natural isomorphism $\Khs{X}{X}{m} \to \Khs{M}{X}{m}$. Since $\psi^k_X$ corresponds to $\psi^k$ under the isomorphism $\Khs{X}{X}{m}\simeq\Kh{X}{m}$, we obtain the claim.
\end{proof}

\begin{definition}
\label{def:compatloc}
Let $\Omega$ be a fixed variety. An \emph{operation compatible with localization} is the data, for every integer $m$, of a natural transformation $l \colon \Zu \otimes \K'_m \to \Zu \otimes \K'_m$ of functors $\Var/\Omega \to \Ab$ (see \ref{sect:var} and \ref{sect:notations}), such that for every closed embedding $Y \hookrightarrow X$ over $\Omega$, the morphisms in the long localization sequence for $(Y,X)$ are compatible with the maps $l^X$, $l^Y$ and $l^{X-Y}$.
\end{definition}

Two examples of such operations are given in Lemmas~\ref{lemm:action} and \ref{lemm:partial}.
\begin{lemma}
\label{lemm:action}
The action of any element of $\Zu \otimes \Kzl(\Omega)$ is an operation compatible with localization.
\end{lemma}
\begin{proof}
Let $X \in \Var/\Omega$, and $Y$ a closed subvariety of $X$. If $\Ec$ is a locally-free coherent $\Oc_\Omega$-module, then the action of $[\Ec]$ on the $\K'$-groups is induced by the exact functor $-\otimes_{\Oc_\Omega} \Ec$.  If $\Fc$ is a coherent $\Oc_X$-module supported on $Y$, then so is the $\Oc_X$-module $\Fc \otimes_{\Oc_\Omega} \Ec$. Thus by the functoriality statement in \cite[\S 5, Theorem~5]{Qui-72}, the action of $[\Ec]$ commutes with the maps on the long localization sequence. Since $\Kzl(\Omega)$ is additively generated by such classes $[\Ec]$, we are done.
\end{proof}

\begin{lemma}
\label{lemm:partial}
The homological Adams operation $\psi_k$ is compatible with localization.
\end{lemma}
\begin{proof}
In view of Proposition~\ref{prop:adams}, \eqref{adams:proper} and \eqref{adams:open}, it only remains to check the compatibility of $\psi_k$ with connecting homomorphisms in the localization exact sequences. 
  
Let $X$ be a variety, $Y$ a closed subvariety. Choose a closed embedding $X \hookrightarrow M$, with $M$ a smooth variety. Then the connecting homomorphism $\Kg{X-Y}{m+1} \to \Kg{Y}{m}$ factors as $\Kg{X-Y}{m+1} \to \Kg{M-Y}{m+1}\to \Kg{Y}{m}$ \cite[\S7, Remark~3.4]{Qui-72}. Here the first map is the push-forward along the closed embedding $X-Y \hookrightarrow M-Y$, hence is compatible with $\psi_k$. The second map can be identified with the connecting homomorphism $\K_{m+1}(M-Y)\to \Khs{M}{Y}{m}$. It is compatible with (cohomological) Adams operations by \cite[Corollary~5.5]{Levine-Lambda}, with the action of $\theta^k(-\Tan_{M}) \in \Kzl(M)$ by Lemma~\ref{lemm:action}, and therefore with $\psi_k$, as requested.
\end{proof}

\begin{proposition}
\label{prop:compatpullback}
Let $f \colon Y \to X$ be a local complete intersection morphism with virtual tangent bundle $\Tan_f \in \Kzh(Y)$. As maps $\Zu \otimes \Kg{X}{m} \to \Zu \otimes \Kg{Y}{m}$,
\[
\theta^k(-\Tan_f) \circ f^* \circ \psi_k=\psi_k \circ f^*.
\]
\end{proposition}
\begin{proof}
We loosely follow the structure of the proof of \cite[Theorem~18.2]{Ful-In-98}.

\emph{(First step)} Assume that we have a cartesian square
\[ \xymatrix{
Y\ar[r]^f \ar[d] & X \ar[d] \\ 
N \ar[r]_g & M
}\]
with $N, M$ smooth varieties, vertical arrows closed embeddings, horizontal arrows smooth morphisms. The pull-back $g^* \colon  \Khs{M}{X}{m} \to \Khs{N}{Y}{m}$ is compatible with the Adams operations with supports (this follows from the naturality statement in \cite[Corollary~5.5]{Levine-Lambda}). We have in $\Kzh(N)$ the relation between tangent bundles $[\Tan_N]=[\Tan_g]+[g^*\Tan_M]$. Thus, as maps $\Zu \otimes \Khs{M}{X}{m}\to \Zu \otimes \Khs{N}{Y}{m}$,
\begin{align*}
      \theta^k(-\Tan_g) \circ g^* \circ  \theta^k(-\Tan_{M}) \circ \psi^k_{M}&=\theta^k(-\Tan_g) \circ \theta^k(-g^*\Tan_M) \circ g^*\circ \psi^k_M\\
      &=\theta^k(-\Tan_N) \circ \psi^k_{N} \circ g^*.
\end{align*}
The requested formula follows in this case by using the identifications $\Khs{N}{Y}{m} \simeq \Kg{Y}{m}$ and $\Khs{M}{X}{m} \simeq \Kg{X}{m}$, since $\Tan_g|_Y=\Tan_f$.

\emph{(Second step)} Assume that $f$ is a vector bundle. By \cite[Lemma~18.2]{Ful-In-98}, we can find $M$ and $N$ as in the first step, so that the formula holds in this case.

\emph{(Third step)} Now assume that $f$ is a regular closed embedding, with normal bundle $p\colon N_f \to Y$. We have $p^*\Tan_f=-p^*[N_f]=-\Tan_p \in \Kzl(Y)$. Applying Lemma~\ref{lemm:def} below, and using the second step for the morphism $p$, we compute
\begin{align*} 
 p^* \circ f^* \circ \psi_k &= \sigma^f \circ \psi_k=\psi_k \circ \sigma^f=\psi_k \circ p^* \circ f^*\\
  &=\theta^k(-\Tan_p)\circ p^*\circ \psi_k \circ f^*\\
  &=p^* \circ \theta^k(\Tan_f)\circ \psi_k \circ f^*.
\end{align*}
The first formula then follows since $p^*$ is injective \cite[Theorem~5.3]{Cai}.

\emph{(Fourth step)} Finally let $f$ be an arbitrary local complete intersection morphism. Choose a closed embeddings $X \hookrightarrow M$ and $Y \hookrightarrow N$ with $M,N$ smooth varieties. Then consider the commutative diagram, 
\[ \xymatrix{
Y \ar[r]^-h & X \times N\ar[r]^-p \ar[d] & X \ar[d] \\ 
&M\times N \ar[r] & M.
}
\]
The closed embedding $h$ is regular by \cite[Chapter~IV, Corollary~3.10]{FL-Ri-85}. We conclude by using the third step for $h$, and the first step for $p$.
\end{proof}

\begin{lemma}
\label{lemm:def}
Let $f \colon Y \hookrightarrow X$ be a closed embedding with normal cone $N_f$. Let $\sigma^f \colon \Kg{X}{m} \to \Kg{N_f}{m}$ be the deformation homomorphism (see \ref{sect:def}). Then we have, as maps $\Zu \otimes \Kg{X}{m} \to \Zu \otimes \Kg{N_f}{m}$,
\[
\sigma^f \circ \psi_k=\psi_k \circ \sigma^f.
\]
\end{lemma}
\begin{proof}
We use the notations of \ref{sect:def}. Let $x \in \Zu \otimes \Kg{X}{m}$. By Lemma~\ref{lemm:partial} and Proposition~\ref{prop:adams}, \eqref{item:extprod}, we have in $\Zu \otimes \Kg{N_f}{m}$
\[
\psi_k \circ \sigma^f(x)=\psi_k \circ \delta^f( \{t\} \times x)=\delta^f \circ \psi_k( \{t\} \times x)=\delta^f( \psi_k\{t\}\times \psi_k(x)).
\]
We conclude the proof with Lemma~\ref{lemm:tk} below.
\end{proof}

\begin{lemma}
\label{lemm:tk}
For the element $\{t\} \in \Kg{\Gm}{1}$ considered in \ref{sect:def}, we have 
\[
\psi_k\{t\}=\{t\}.
\]
\end{lemma}
\begin{proof}
By Proposition~\ref{prop:regular} and \eqref{eq:thetakn}, we have in $\Zu \otimes \Kg{\Gm}{1}$
\[
\psi_k\{t\}=\theta^k(-\Tan_{\Gm}) \circ \psi^k\{t\}=\theta^k(-1) \circ \psi^k\{t\}=k^{-1} \cdot \psi^k\{t\}=\ k^{-1} \cdot \{t^k\}=\{t\}.\qedhere
\]
\end{proof}

\section{Lifting operations from $K$ to $CK$}
We fix a variety $\Omega$. Let $X \in \Var/\Omega$ (see \ref{sect:var}), and $p$ an integer. Consider the category $\Cc_p$ with objects pairs $Z \subset W$ of closed subvarieties of $X$, such that $\dim W \leq p$ and $\dim Z \leq p-1$. A morphism $(Z',W') \to (Z,W)$ is a pair of closed embeddings $Z' \hookrightarrow Z$ and $W' \hookrightarrow W$ (over $X$). Thus $\Homm_{\Cc_p}( (Z',W') , (Z,W) )$ has at most one element, and the category $\Cc_p$ is filtered.

Given a morphism $i\colon (Z',W') \to (Z,W)$, we let $u \colon W-Z \to W-Z'$ be the open embedding, and $j \colon  W'-Z' \hookrightarrow W-Z'$ the closed embedding. We set
\[
i_*=u^* \circ j_* \colon \Kg{W'-Z'}{p+q} \to \Kg{W-Z}{p+q}.
\]
This makes $(Z,W) \mapsto \Kg{W-Z}{p+q}$ a functor $\Cc_p \to \Ab$. The associations $(Z,W) \mapsto \Kg{Z}{p+q}$ and $(Z,W) \mapsto \Kg{W}{p+q}$ also define such functors.

For every $(Z,W) \in \Cc_p$, we have a long localization sequence
\[
\cdots \to \Kg{W-Z}{p+q+1} \to \Kg{Z}{p+q} \to \Kg{W}{p+q} \to \Kg{W-Z}{p+q} \to \cdots
\]
Taking the filtering colimit over $\Cc_p$, and using the notation \eqref{eq:e1}, we obtain the exact sequence \cite[\S7, 5.]{Qui-72}
\begin{equation}
\label{seq:longexact}
\cdots \to E^1_{p,q+1}  \xrightarrow{\delta} \K_{p+q}(\Mc_{p-1}(X)) \xrightarrow{b} \K_{p+q}(\Mc_p(X)) \xrightarrow{f} E^1_{p,q} \to \cdots,
\end{equation}
Tensoring with $\Zu$, and taking the colimit over $\Cc_p$, we see that any operation $l$ compatible with localization (Definition~\ref{def:compatloc}) induces a commutative diagram 
\begin{equation}
\label{seq:longexact2}
\xymatrix{
E^1_{p,q+1}[\frac{1}{k}]\ar[r]^-{\delta} \ar[d]^{l^{(p/p-1,q+1)}} &  \K_{p+q}(\Mc_{p-1}(X))[\frac{1}{k}] \ar[d]^{l^{(p-1,q+1)}}\ar[r]^-b& \K_{p+q}(\Mc_p(X))[\frac{1}{k}]\ar[d]_{l^{(p,q)}} \ar[r]^-f&E^1_{p,q}[\frac{1}{k}] \ar[d]_{l^{(p/p-1,q)}}\\ 
E^1_{p,q+1}[\frac{1}{k}]\ar[r]_-{\delta} &  \K_{p+q}(\Mc_{p-1}(X))[\frac{1}{k}] \ar[r]_-b &\K_{p+q}(\Mc_p(X))[\frac{1}{k}]\ar[r]_-f&E^1_{p,q}[\frac{1}{k}]
}
\end{equation}

\begin{definition}
\label{def:gentrivial}
Let $n,q$ be integers. Assume that we have, for every integer $m\leq n$, a natural transformation $l \colon \Zu \otimes \K'_m \to \Zu \otimes \K'_m$ of functors $\Var/\Omega \to \Ab$ (see \ref{sect:var} and \ref{sect:notations}). We say that $l$ is \emph{generically trivial in weight $q$ and degree $\leq n$}, if for every locally integral variety $X$ over $\Omega$, and every generic point $x \colon \Spec(\kappa(x)) \to X$, we have, when $\dim_x X +q \leq n$ 
\[
x^* \circ l^X=0 \colon \Zu \otimes \Kg{X}{\dim_xX +q} \to \Zu \otimes\Kg{\Spec(\kappa(x))}{\dim_xX +q},
\]
where $\dim_x X$ is the dimension of the connected component containing $x$.
\end{definition}

Two examples of such operations can be found in Propositions~\ref{prop:adgen} and \ref{prop:action}.

\begin{lemma}
\label{lemm:vanish}
Let $l$ be an operation compatible with localization (Definition~\ref{def:compatloc}) and generically trivial in weight $q$ and degree $\leq n$ (Definition~\ref{def:gentrivial}). If $p \leq n-q$,
\[
l^{(p/p-1,q)}=0 \colon \Zu \otimes E^1_{p,q} \to \Zu \otimes E^1_{p,q}.
\]
\end{lemma}
\begin{proof}
Given $(Z,W) \in \Cc_{p}$, consider the morphism $r \colon (Z_{red},W_{red}) \to (Z,W)$ given the reduced structures. The maps $r_*$ give, by d\'evissage, an isomorphism between the localization sequences of the two pairs. Thus, while computing a colimit over $\Cc_p$, we can restrict our attention to pairs of reduced subvarieties. Moreover for such a pair $(Z,W)$, after enlarging $Z$, we can assume that $W-Z$ is locally integral. The restriction to the generic points map gives a flat pull-back
\[
g_{W-Z}^*\colon \Kg{W-Z}{p+q} \to \coprod_{x \in X_{(p)}\cap (W-Z)} \Kg{\Spec(\kappa(x))}{p+q}.
\]
whose colimit is the identity of $E^1_{p,q}$. Since $l$ is generically trivial in weight $q$ and degree $\leq n$, we have $g_{W-Z}^* \circ l^{W-Z}=0$ for $p+q \leq n$. The statement follows.
\end{proof}

We will use the notation 
\[
\Ck_{p,q}=\Zu \otimes \CK_{p,q} \colon \Var/\Omega \to \Ab.
\]
We have a commutative diagram of functors $\Var/\Omega \to \Ab$:
\[ 
\xymatrix{
\Zu \otimes \K_{p+q}(\Mc_p(-))\ar[rr]^j \ar[d]_b && \Ck_{p,q}(-) \ar[d]^{\beta}\ar[dll]^i \\ 
\Zu \otimes \K_{p+q}(\Mc_{p+1}(-)) \ar[rr]_j && \Ck_{p+1,q-1}(-)
}
\]
The map $i$ is a monomorphism, and $j$ is an epimorphism. When proving equality of type $f=g$ for maps $f,g$ between $\Ck$-groups, we will without further notice simply prove that $i \circ f \circ j=i \circ g \circ j$.

\begin{proposition}
\label{prop:varphi}
Let $l$ be an operation compatible with localization (Definition~\ref{def:compatloc}). There is a unique natural transformation $\lambda \colon \Ck_{p,q} \to \Ck_{p,q}$ such that
\[
\lambda \circ j=j \circ l^{(p,q)} \quad \text{ and } \quad i \circ \lambda=l^{(p+1,q-1)} \circ i.
\]
\end{proposition}
\begin{proof}
Unicity is immediate, and existence follows from the relation (see \eqref{seq:longexact2})
\[
l^{(p+1,q-1)} \circ i \circ j=i \circ j \circ l^{(p,q)}.\qedhere
\]
\end{proof}

\begin{proposition}
\label{prop:sigma}
In the situation of Proposition~\ref{prop:varphi}, assume additionally that $l$ is generically trivial in weight $q$ and degree $\leq n$ (Definition~\ref{def:gentrivial}). Let $p \leq n-1-q$. There is a unique natural transformation $\Lambda \colon \Ck_{p,q} \to \Ck_{p-1,q+1}$ such that
\[
i \circ \Lambda \circ j=l^{(p,q)} \quad \text{ and } \quad \lambda=\beta \circ \Lambda.
\]
\end{proposition}
\begin{proof}
We use the notations of \eqref{seq:longexact2}. Since $p+q \leq n$ (resp.\ $p+q \leq n-1$), we know by Lemma~\ref{lemm:vanish} that $l^{(p/p-1,q)} =0$ (resp.\ $l^{(p+1/p,q)} =0$). Hence
\[
f \circ l^{(p,q)}=l^{(p/p-1,q)} \circ f \quad (\text{resp.\ } l^{(p,q)} \circ \delta=\delta \circ l^{(p+1/p,q)})
\]
is zero, and 
\[
\im l^{(p,q)} \subset \ker f=\im b=\im i \quad (\text{resp.\ } \ker j=\ker b=\im \delta \subset \ker l^{(p,q)}).
\]
This proves existence of the operation $\Lambda$ satisfying the first condition (unicity is immediate). The second formula follows from the computation
\[
i \circ \lambda \circ j=i\circ j \circ l^{(p,q)}=i\circ j \circ i \circ \Lambda \circ j=i \circ \beta \circ \Lambda \circ j.\qedhere
\]
\end{proof}

We now investigate the behaviour of the operations $\lambda$ and $\Lambda$ with respect to the connecting homomorphisms in the localization sequences for $\CK$-groups.

Let  $Y \hookrightarrow X$ be a closed embedding. We view $\Mc_p(Y)$ as a subcategory of $\Mc_p(X)$, and the usual argument involving d\'evissage and localization yields a long exact sequence. Let $\Delta \colon \K_{p+q}(\Mc_p(X-Y)) \to \K_{p+q-1}(\Mc_p(Y))$ be the connecting homomorphism in this sequence. It can be identified with the colimit of the connecting homomorphisms $\Kg{Z-Z\cap Y}{p+q} \to \Kg{Z\cap Y}{p+q-1}$ in the localization sequences associated with the closed embeddings $Z\cap Y \hookrightarrow Z$, where $Z$ runs over the closed subvarieties of $X$, with $\dim Z \leq p$. Thus, if $l$ is an operation compatible with localization, we have, as maps $\K_{p+q}(\Mc_p(X-Y))[\frac{1}{k}] \to \K_{p+q-1}(\Mc_p(Y))[\frac{1}{k}]$
\begin{equation}
\label{eq:deltal}
\Delta \circ l^{(p,q)}=l^{(p,q-1)} \circ \Delta.
\end{equation}

By construction of the connecting homomorphism $\partial \colon \CK_{p,q}(X-Y) \to \CK_{p-1,q}(Y)$ \cite[Theorem~5.1]{Cai}, we have, as maps $\K_{p+q}(\Mc_p(X-Y)) \to \K_{p+q-1}(\Mc_p(Y))$,
\begin{equation}
\label{eq:deltapartial}
i \circ \partial \circ j=\Delta.
\end{equation}
From the equalities of maps $\K_{p+q}(\Mc_p(X-Y)) \to \K_{p+q-1}(\Mc_{p+1}(Y))$
\[
i \circ \partial \circ \beta \circ j=\Delta \circ b=b\circ \Delta=i \circ \beta \circ \partial \circ j,
\]
We deduce the equality of maps $\CK_{p,q}(X-Y) \to \CK_{p,q+1}(Y)$
\begin{equation}
\label{eq:deltabeta}
\partial \circ \beta=\beta \circ \partial.
\end{equation}

\begin{proposition}
\label{prop:varphipartial}
Let $Y \hookrightarrow X$ be a closed embedding, $\partial \colon \Ck_{p,q}(X-Y) \to \Ck_{p-1,q}(Y)$ the corresponding connecting homomorphism. In the situation of Proposition~\ref{prop:varphi}, we have, as maps $\Ck_{p,q}(X-Y) \to \Ck_{p-1,q}(Y)$,
\[
\partial \circ \lambda=\lambda \circ \partial. 
\]
\end{proposition}
\begin{proof}
By \eqref{eq:deltal} and \eqref{eq:deltapartial}, we have, as maps $\K_{p+q}(\Mc_p(X-Y))[\frac{1}{k}] \to \K_{p+q-1}(\Mc_{p-1}(Y))[\frac{1}{k}]$,
\[
i \circ \partial \circ \lambda \circ j=i \circ \partial \circ j \circ l^{(p,q)}=\Delta \circ l^{(p,q)}=l^{(p,q-1)} \circ \Delta=l^{(p,q-1)} \circ i \circ \partial \circ j=i \circ \lambda \circ \partial \circ j.\qedhere
\]
\end{proof}

\begin{proposition}
\label{prop:sigmapartial}
Let $Y \hookrightarrow X$ be a closed embedding, $\partial \colon \Ck_{p,q}(X-Y) \to \Ck_{p-1,q}(Y)$ the corresponding connecting homomorphism. In the situation of Proposition~\ref{prop:sigma}, we have, as maps $\Ck_{p,q}(X-Y) \to \Ck_{p-1,q}(Y)$ (with  $p\leq n-1-q$),
\[
\beta \circ \Lambda \circ \partial = \beta \circ \partial \circ \Lambda. 
\]
\end{proposition}
\begin{proof}
This follows from Propositions~\ref{prop:sigma} and \ref{prop:varphipartial}, and from \eqref{eq:deltabeta}.
\end{proof}

\section{Some operations in connective $K$-theory}
Combining Proposition~\ref{prop:varphi} and Lemma~\ref{lemm:partial}, we obtain, in the notations of \eqref{seq:longexact2}
\begin{proposition}
\label{prop:varphik}
There is a unique natural transformation of functors $\Var \to \Ab$
\[
\varphi_k \colon \Ck_{p,q} \to \Ck_{p,q}
\]
such that
\[
\varphi_k \circ j=j \circ \psi_k^{(p,q)} \quad \text{ and } \quad i \circ \varphi_k=\psi_k^{(p+1,q-1)} \circ i.
\]
\end{proposition}

These operations can be thought of as homological Adams operations for connective $K$-theory (see also Proposition~\ref{prop:lci}, \eqref{item:lci:phi}). Over a field of characteristic zero, after tensoring with $\mathbb{Q}$, similar operations were considered in \cite{MalagonLopez}.\\

Our next purpose is to modify $\varphi_k$ in order to obtain a generically trivial operation. Therefore we need to investigate the behaviour of homological Adams operations with respect to restriction to a generic point.

\begin{lemma}
\label{lemm:mult}
Let $X$ be a locally integral variety, $x \colon \Spec(\kappa(x)) \to X$ a generic point. Let $m \leq 2$. Then, as maps $\Zu \otimes \Kg{X}{m} \to \Zu \otimes \Kg{\Spec(\kappa(x))}{m}$,
\[
x^* \circ \psi_k=k^{m-\dim_x X} \cdot x^*.
\]
\end{lemma}
\begin{proof}
Let $u \colon U \to X$ be a non-empty open connected regular subvariety of $X$ containing $x$. Let $g \colon \Spec(\kappa(x)) \to U$ be the generic point of $U$. Then, using Proposition~\ref{prop:adams},~\eqref{adams:open}, Proposition~\ref{prop:regular}, and the relation \eqref{eq:thetakn},  
\begin{align*} 
x^* \circ \psi_k &= g^* \circ u^*\circ \psi_k \\
	&=g^* \circ  \psi_k \circ u^*\\
    &=g^* \circ \theta^k(-\Tan_U) \circ \psi^k \circ u^*\\
    &=\theta^k(-g^*\Tan_U) \circ \psi^k \circ g^* \circ u^*\\
    &=k^{-\dim_x X} \cdot \psi^k \circ x^*.
\end{align*}
The $(m+1)$-st term of the gamma filtration $\Figh{m+1}{\kappa(x)}{m}$ vanishes \cite[Th\'eor\`eme~1]{Sou-Op-85}. Since $m \leq 2$, we have $\Kg{\Spec(\kappa(x))}{m}=\Kh{\kappa(x)}{m}=\Figh{m}{\kappa(x)}{m}$ \cite[\S 2]{Sou-Op-85}. Thus $\psi^k$ acts on $\K_m(\kappa(x))$ by multiplication by $k^m$, when $m\leq2$. (Alternatively, this is true for $\Kh{\kappa(x)}{0}=\Z$ and $\Kh{\kappa(x)}{1}=\kappa(x)^{\times}$; since $\Kh{\kappa(x)}{2}$ is generated by products of two elements of $\Kh{\kappa(x)}{1}$, and $\psi^k$ is multiplicative, we see that this is also true for $m=2$.)
\end{proof}

\begin{proposition}
\label{prop:adgen}
The operation $\psi_k-k^q \id$ is compatible with localization (Definition~\ref{def:compatloc}) and generically trivial in weight $q$ and degree $\leq 2$ (Definition~\ref{def:gentrivial}).
\end{proposition}
\begin{proof}
The first statement follows from Lemma~\ref{lemm:partial}, since multiplication by $k^q$ is compatible with localization. The second statement follows from Lemma~\ref{lemm:mult}.
\end{proof}

Combining Propositions \ref{prop:sigma} and \ref{prop:adgen}, we obtain, in the notations of \eqref{seq:longexact2}
\begin{proposition}
\label{prop:sigmak}
There is a unique natural transformation of functors $\Var \to \Ab$
\[
\sik \colon \Ck_{p,q} \to \Ck_{p-1,q+1} \quad (p+q \leq 1),
\]
such that
\[
i \circ \sik \circ j=\psi_k^{(p,q)} - k^q \id.
\]
\end{proposition}

\begin{proposition}
\label{prop:extprod}
Let $X$ and $Y$ be varieties. Let $x \in \Ck_{p,q}(X)$ and $y \in \Ck_{r,s}(Y)$, with $p+q+r+s\leq 1$. Then, in $\Ck_{p+r-1,q+s+1}(X\times Y)$,
\[
\sik(x \times y)=k^s \cdot \sik(x) \times y + k^q \cdot x \times \sik(y) + \beta \big( \sik(x) \times \sik(y) \big).
\]
\end{proposition}
\begin{proof}
We compute, for $a,b$ such that $j(a)=x$, $j(b)=y$,
\begin{align*}
   &i \circ \sik \circ j(a \times b)=\psi_k^{(p,q)}(a) \times \psi_k^{(r,s)}(b) - k^{q+s} a \times b\\
   &=k^s(\psi_k^{(p,q)}(a)-k^qa) \times b + k^q a \times (\psi_k^{(r,s)}(b) -k^sb) + ( \psi_k^{(p,q)}(a)-k^qa) \times (\psi_k^{(r,s)}(b)-k^sb)\\
   &=k^s i\circ \sik \circ j(a)\times b + k^q a \times i\circ \sik \circ j(b) + i\circ \sik \circ j(a) \times i\circ \sik \circ j(b)\\
   &=i\circ \Big( k^s \sik \times \id + k^q \id \times \sik + \beta \circ \sik\times \sik\Big) \circ j(a \times b).
\end{align*}

The last equality follows from the formulas
\[
i(u) \times i(v)=i\circ \beta (u \times v) \quad \text{ and } \quad j(u) \times j(v)=j(u\times v).\qedhere
\]
\end{proof}

\begin{lemma}
\label{lemm:commutes}
Assume that $p+q \leq 1$. Then, as maps $\Ck_{p,q} \to \Ck_{p,q}$
\[
\beta \circ \sik - \sik \circ \beta=(1-k) k^{q-1} \cdot\id.
\]
\end{lemma}
\begin{proof}
Recall from \eqref{seq:longexact2} that we have 
\[
i \circ j \circ \psi_{k}^{(p,q)}=\psi_{k}^{(p+1,q-1)} \circ i \circ j.
\]
We compute
\begin{align*}
i \circ (\beta \circ \sik -\sik \circ \beta) \circ j&=i\circ j\circ i \circ \sik \circ j - i \circ \sik \circ j \circ i   \circ j\\
  &=i \circ j \circ (\psi_{k}^{(p,q)} - k^q \id) - (\psi_{k}^{(p+1,q-1)} - k^{q-1}\id) \circ i \circ j\\
  &=-k^q \cdot i \circ j + k^{q-1} \cdot i \circ j\\
  &=i\circ( (1-k)k^{q-1} \cdot \id ) \circ j. \qedhere
\end{align*}
\end{proof}

We now introduce a characteristic class $\pi^k$ for $\Ck$-groups. This will allow us to express the compatibility of the operation $\sik$ with pull-backs in the next section.
\begin{proposition}
\label{prop:action}
Let $\Omega$ be a variety, and $u \in \Zu \otimes \Kzh(\Omega)$ such that $\rank(u)=0$. Then the action of $u$ is compatible with localization (Definition~\ref{def:compatloc}) and generically trivial in weight $q$ and degree $\leq n$ (Definition~\ref{def:gentrivial}), for every $q,n$.
\end{proposition}
\begin{proof}
The first statement is Lemma~\ref{lemm:action}. Let $f \colon X \to \Omega$ be a morphism, with $X$ a locally integral variety. If $x \colon \Spec(\kappa(x)) \to X$ is a generic point, then $x^*\colon \Zu \otimes \Kg{X}{m} \to \Zu \otimes \Kh{\kappa(x)}{m}$ respects the actions of $u$. But $u$ acts on $\Zu \otimes \Kh{\kappa(x)}{m}$ by multiplication by the element $x^* \circ f^*(u) \in \Zu \otimes \Kh{\kappa(x)}{0}$. The latter has rank zero, and the rank map $\Kzl(\kappa(x))\to \Z$ is an isomorphism.
\end{proof}

\begin{lemma}
\label{lemm:characteristic}
For every element $y \in \Kzh(X)$, there is a unique morphism
\[
\pi^k(y) \colon \Ck_{p,q}(X) \to \Ck_{p-1,q+1}(X)
\]
satisfying
\[
i \circ \pi^k(y) \circ j=\theta^k(y) - k^{\rank(y)} \id.
\]
Moreover $\pi^k$ is a characteristic class, i.e. if $f \colon Y\to X$ is a flat morphism, then 
\[
\pi^k(f^*(y))\circ f^*=f^* \circ \pi^k(y).
\]
\end{lemma}
\begin{proof}
The second statement follows from the first, and from the fact that $\theta^k(y) - k^{\rank(y)} \id$ is a characteristic class. We obtain the existence of $\pi^k$ by using \eqref{eq:thetakn}, Propositions~\ref{prop:action} and \ref{prop:sigma} for $\Omega=X$.
\end{proof}

The class $\pi^k$ does not take sums to composition of morphisms, but we have
\begin{lemma}
\label{lemm:pitheta}
Let $x,y \in \Kzl(X)$. Then, as maps $\Ck_{p,q}(X) \to \Ck_{p-1,q+1}(X)$,
\[
\pi^k(x+y)=\theta^k(x) \circ \pi^k(y) + k^{\rank(y)}\cdot \pi^k(x).
\]
\end{lemma}
\begin{proof}
We compute, as maps $\Zu \otimes \K_{p+q}(\Mc_p(X)) \to \Zu \otimes \K_{p+q}(\Mc_p(X))$,
\begin{align*} 
    i \circ \pi^k(x+y) \circ j &= \theta^k(x+y) - k^{\rank(x+y)}\cdot\id\\
    &=\theta^k(x) \circ \theta^k(y)-k^{\rank(x)} k^{\rank(y)}\cdot\id\\
    &=\theta^k(x)\circ\big(\theta^k(y)-k^{\rank(y)}\cdot \id\big)+k^{\rank(y)}\cdot\big(\theta^k(x)-k^{\rank(x)}\cdot \id\big)\\
    &=\theta^k(x)\circ i \circ \pi^k(y)\circ j+k^{\rank(y)}\cdot i \circ \pi^k(x)\circ j\\
    &=i \circ \big(\theta^k(x) \circ \pi^k(y) + k^{\rank(y)}\cdot \pi^k(x)\big) \circ j.\qedhere
\end{align*}
\end{proof}

\section{Compatibility with Gysin maps}
The purpose of this section is to describe the behaviour of the operations
\[
\varphi_k \colon \Ck_{p,q} \to \Ck_{p,q} \quad \text{ and } \quad \sik \colon \Ck_{p,q} \to \Ck_{p-1,q+1}
\]
of Propositions~\ref{prop:varphik} and \ref{prop:sigmak} with respect to pull-backs along local complete intersection morphisms. While proving the main statement of this section (Proposition~\ref{prop:lci}) for the operation $\sik$ as above, we shall use the existence of the operation
\[
\sik \colon \Ck_{p+1,q} \to \Ck_{p,q+1},
\]
which is defined only for $q=-p$. Therefore we shall make this assumption then. Note that this is not really harmful since we are ultimately interested in the Chow group.\\

We now formalize an argument that will be used in various situations.
\begin{definition}
Let $\Omega$ be a variety. We say that a functor $\Fc\colon \Var/\Omega \to \Var/\Omega$ is of \emph{relative dimension $\leq d$} if for every $Z \in \Var/\Omega$, we have $\dim \Fc(Z)\leq \dim Z +d$.
\end{definition}

\begin{example}
\label{ex:flat}
(See \cite[B.2.5]{Ful-In-98}) Let $f\colon Y \to X$ be a flat morphism of relative dimension $d$. Then the functor $\Fc_f\colon \Var/X \to  \Var/X, Z \mapsto Z \times_X Y$ defined by taking the inverse image under $f$ has relative dimension $\leq d$.
\end{example}

\begin{example}
\label{ex:normalcone}
Let $f \colon Y \hookrightarrow X$ be a closed embedding. Let $Z \in \Var/X$. We define $\Nn_f(Z)$ as the normal cone of the induced closed embedding $Z\times_X Y \hookrightarrow Z$. If $Z' \to Z$ is a proper morphism of schemes over $X$, there is a natural proper morphism $\Nn_f(Z') \to \Nn_f(Z)$ over $X$. This gives a functor $\Nn_f \colon \Var/X \to \Var/X$. It follows from \cite[B.6.6]{Ful-In-98} that $\Nn_f$ has relative dimension $\leq 0$.
\end{example}

We refer to \ref{sect:by} for the definition of the map $\by \colon \Ck_{p,q} \to \K'_{p+q}$.
\begin{proposition}
\label{prop:zero}
Let $\Omega$ be a variety. Let $\mathcal{F} \colon \Var/\Omega \to \Var/\Omega$ be a functor of relative dimension $\leq d$. Let $\eta \colon \Ck_{p,q} \to \Ck_{p+d-1,s} \circ \mathcal{F}$ be a natural transformation of functors $\Var/\Omega \to \Ab$. Assume that $\by \circ \eta=0$. Then $\eta=0$.
\end{proposition}
\begin{proof}
Let $X \in \Var/\Omega$. We prove that $\eta^X \colon \Ck_{p,q}(X) \to \Ck_{p+d-1,s}(\Fc(X))$ vanishes. Recall that we have \cite[\S7, (5.1)]{Qui-72}
\[
\colim{z} \Kg{Z}{p+q} \simeq \K_{p+q}(\Mc_p(X)),
\]
where $z$ runs over the closed embeddings $z \colon Z \hookrightarrow X$, with $\dim Z \leq p$. It follows that the group $\Ck_{p,q}(X)$ is generated by the subgroups $z_*\Ck_{p,q}(Z)$, with $z$ as above. Since $\Fc$ has relative dimension $\leq d$, the variety $\Fc(Z)$ has dimension $\leq p+d$, hence $\by \colon \Ck_{p+d-1,s}(\mathcal{F}(Z)) \to \Kg{\mathcal{F}(Z)}{p+d-1+s}$ is injective. It follows that $\eta^Z(\Ck_{p,q}(Z))=0$. We are done since $\eta^X \circ z_*=\Fc(z)_* \circ \eta^Z$.
\end{proof}

\begin{remark}
The proof actually shows that the category $\Var/\Omega$ can be replaced, in Proposition~\ref{prop:zero}, by its subcategory where morphisms are closed embeddings.
\end{remark}

\begin{lemma}
\label{lemm:t}
Let $Y$ be a variety, and $x \in \Ck_{p,q}(Y)$.
\begin{enumerate}[a)]
\item \label{item:t:phi} We have $\varphi_k(\{t\} \times x)=\{t\} \times \varphi_k(x)$ in $\Ck_{p+1,q}(\Gm \times Y)$.

\item \label{item:t:sigma} When $p+q=0$, we have $\sik(\{t\} \times x)=\{t\} \times \sik(x)$ in $\Ck_{p,q-1}(\Gm \times Y)$.
\end{enumerate}
\end{lemma}
\begin{proof}
Under the identification $\K_1(\Mc_1(\Gm))=\Kg{\Gm}{1}$, the operation $\psi_k^{(1,0)}$ corresponds to $\psi_k$, so that Lemma~\ref{lemm:tk} gives
\[
\psi_k^{(1,0)}\{t\}=\{t\}.
\]
For any $y \in \Ck_{p,q}(Y)$ we have $i(\{t\} \times y)=\{t\} \times i(y)$, and
\[
i \circ \varphi_k(\{t\} \times x)=\psi_k^{(p+2,q-1)} (\{t\} \times i(x))=\psi_k^{(1,0)}\{t\} \times \psi_k^{(p+1,q-1)} \circ i(x)=i (\{t\} \times \varphi_k (x)).
\]
This proves \eqref{item:t:phi}. Assume that $p+q=0$. We have
\[
i \circ \sik\{t\}=i\circ \sik \circ j\{t\}=(\psi_k^{(1,0)} - k^0\id)\{t\}=0.
\]
We conclude the proof of \eqref{item:t:sigma} using Proposition~\ref{prop:extprod} applied with $X=\Gm$.
\end{proof}

\begin{lemma}
\label{lemm:deform}
Let $f \colon Y \hookrightarrow X$ be a closed embedding with normal cone $N_f$. 
\begin{enumerate}[a)]
\item \label{item:deform:phi} We have the equality of maps $\Ck_{p,q}(X) \to \Ck_{p,q}(N_f)$,
\[
\sigma^f \circ \varphi_k=\varphi_k \circ \sigma^f.
\]

\item \label{item:deform:sigma} We have the equality of maps $\Ck_{p,-p}(X) \to \Ck_{p-1,1-p}(N_f)$,
\[
\sigma^f \circ \sik=\sik \circ \sigma^f.
\]
\end{enumerate}
\end{lemma}
\begin{proof}
We use the notations of \ref{sect:def}. We have for $x \in \Ck_{p,q}(X)$, by Lemma~\ref{lemm:t}, \eqref{item:t:phi}, and Proposition~\ref{prop:varphipartial},
\[
\sigma^f \circ \varphi_k(x)=\partial^f(\{t\}\times \varphi_k(x))=\partial^f\circ  \varphi_k(\{t\} \times x)=\varphi_k \circ\partial^f(\{t\} \times x)=\varphi_k \circ \sigma^f(x),
\]
proving \eqref{item:deform:phi}.

In order to obtain \eqref{item:deform:sigma}, we use the functor $\Nn_f$ of Example~\ref{ex:normalcone}. For $Z \in \Var/X$, we denote by $g \colon Z\times_X Y \hookrightarrow Z$ the induced closed embedding, whose normal cone is $\Nn_f(Z)$. Consider the natural transformation $\eta \colon \Ck_{p,-p} \to \Ck_{p-1,1-p} \circ \Nn_f$ of functors $\Var/X \to \Ab$ such that, for $Z$ as above,
\[
\eta^Z=\sik \circ \sigma^g-\sigma^g\circ \sik.
\]
We have for $x \in \Ck_{p,-p}(Z)$, using Proposition~\ref{prop:sigmapartial} and Lemma~\ref{lemm:t}, \eqref{item:t:sigma}
\begin{align*} 
\beta \circ \sik \circ \sigma^g(x)&=\beta \circ \sik \circ \partial^g (\{t\} \times x)\\
&=\beta \circ \partial^g \circ \sik (\{t\} \times x)\\
&=\beta \circ \partial^g(\{t\} \times \sik(x))\\
&=\beta \circ\sigma^g \circ \sik(x).
\end{align*}
Thus $\beta \circ \eta^Z=0$, and \eqref{item:deform:sigma} follows from Proposition~\ref{prop:zero}, applied with $\Omega=X$.
\end{proof}

We are now in position to prove the main result of this section.
\begin{proposition}
\label{prop:lci}
Let $f \colon Y \to X$ be a local complete intersection morphism of constant relative dimension $d$, with virtual tangent bundle $\Tan_f$.
\begin{enumerate}[a)]
\item \label{item:lci:phi} We have the equality of maps $\Ck_{p,q}(X) \to \Ck_{p+d,q-d}(Y)$,  
\[
\theta^k(-\Tan_f) \circ f^* \circ \varphi_k=\varphi_k \circ f^*.
\]

\item \label{item:lci:sigma} We have the equality of maps $\Ck_{p,-p}(X) \to \Ck_{p+d-1,1-d-p}(Y)$,
\[
f^* \circ \sik=\theta^k(\Tan_f) \circ \sik \circ f^* + k^{-p-d} \cdot \pi^k(\Tan_f)\circ f^*
\]
\end{enumerate}
\end{proposition}
\begin{proof}[Proof when $f$ is smooth.]
Let $\Fc_f$ be the functor of Example~\ref{ex:flat}. For $Z \in \Var/X$, we denote by $f_Z \colon \Fc_f(Z)=Z\times_X Y \to Z$ the induced map. We define a natural transformation $\eta \colon \Ck_{p,q} \to \Ck_{p+d,q-d} \circ \Fc_f$ of functors $\Var/X \to \Ab$ by setting
\[
\eta^Z=\theta^k(-\Tan_f) \cdot {f_Z}^* \circ \varphi_k-\varphi_k \circ {f_Z}^*.
\]
We have $\by \circ \varphi_k=\psi_k\circ \by$, and it follows from Propositions~\ref{prop:properflataction} and \ref{prop:compatpullback} that $\by \circ \eta=0$. Since the functor $\Fc_f$ has relative dimension $\leq d+1$, we obtain \eqref{item:lci:phi} with Proposition~\ref{prop:zero}.

Using again Propositions~\ref{prop:properflataction} and \ref{prop:compatpullback}, we have, as maps $\Ck_{p,-p}(X) \to \Kg{Y}{0}$,
\begin{align*} 
 \by \circ f^* \circ \sik &= f^* \circ (\psi_k - k^{-p}\id) \circ \by\\
  &=\theta^k(\Tan_f) \circ \psi_k \circ f^*\circ\by -k^{-p} \cdot f^*\circ \by\\
  &=\theta^k(\Tan_f) \circ (\psi_k -k^{-p-d}\id)\circ f^*\circ\by + k^{-p-d}\cdot \big(\theta^k(\Tan_f)-k^d \id\big)\circ f^*\circ\by\\
  &=\by \circ \big(\theta^k(\Tan_f)\circ \sik + k^{-p-d}\cdot\pi^k(\Tan_f)\big) \circ f^*,
\end{align*}
and \eqref{item:lci:sigma} follows from Proposition~\ref{prop:zero}, since $\Fc_f$ has relative dimension $\leq d$.
\end{proof}
\begin{proof}[Proof when $f$ is a regular closed embedding.] Let $g \colon N_f \to Y$ be its normal bundle. The morphism $g$ is smooth of constant relative dimension $-d$, with virtual tangent bundle $\Tan_g=g^*[N_f]=-g^*\Tan_f$. We have, using Lemma~\ref{lemm:deform}, \eqref{item:deform:phi}, and the already established formula for the smooth morphism $g$,
  \begin{align*} 
   g^* \circ f^*\circ \varphi_k &= \sigma^f \circ \varphi_k=\varphi_k \circ \sigma^f=\varphi_k \circ g^*\circ f^*\\ 
    &=\theta^k(-\Tan_g) \circ g^* \circ \varphi_k \circ f^*\\
    &=g^* \circ \theta^k(\Tan_f)\circ \varphi_k \circ f^*.
     \end{align*}
The first formula then follows since $g^*\colon \Ck_{p+d,q-d}(Y) \to \Ck_{p,q}(N_f)$ is injective \cite[Theorem~5.3]{Cai}.

We have, by Lemma~\ref{lemm:deform}, \eqref{item:deform:sigma}, the already established formula for the smooth morphism $g$, and Lemma~\ref{lemm:pitheta} applied with $x=\Tan_f$ and $y=-\Tan_f$ (noting that $\pi^k(0)=0$), as maps $\Ck_{p,-p}(X) \to \Ck_{p-1,1-p}(N_f)$,
\begin{align*} 
g^* \circ f^* \circ \sik&=\sigma^f \circ \sik=\sik \circ \sigma^f=\sik \circ g^*\circ f^*\\
 &=\theta^k(-\Tan_g) \circ g^* \circ \sik \circ f^* - k^{-p-d-(-d)} \cdot \theta^k(-\Tan_g) \circ \pi^k(\Tan_g) \circ g^* \circ f^*\\
 &=g^* \circ \theta^k(\Tan_f) \circ \sik \circ f^* - k^{-p} \cdot g^*\circ \theta^k(\Tan_f) \circ \pi^k(-\Tan_f)\circ f^*\\
 &=g^* \circ \big(\theta^k(\Tan_f) \circ \sik +  k^{-p-d} \cdot \pi^k(\Tan_f)\big) \circ f^*.
\end{align*}
We conclude as above.
\end{proof}
\begin{proof}[Proof of Proposition~\ref{prop:lci}.] Decompose $f$ as $j \circ i$, with $i$ a regular closed embedding of constant relative dimension $d_i$, and $j$ a smooth morphism of constant relative dimension $d_j$. Then $\Tan_f=\Tan_i + i^*\Tan_j$, and $d=d_i+d_j$. The formula \eqref{item:lci:phi} follows from the previously considered cases:
\[
\varphi_k \circ f^*=\theta^k(-\Tan_i)\circ i^*\circ \varphi_k \circ j^*=\theta^k(-\Tan_i)\circ i^*\circ \theta^k(-\Tan_j) \circ j^* \circ \varphi_k=\theta^k(-\Tan_f) \circ f^* \circ \varphi_k.
\]

To obtain \eqref{item:lci:sigma}, we also use both cases considered above, and Lemma~\ref{lemm:pitheta} applied with $x=i^*\Tan_j$ and $y=\Tan_i$, to compute, as maps $\Ck_{p,-p}(X) \to \Ck_{p+d-1,1-p-d}(Y)$,
\begin{align*} 
   f^* \circ \sik &= i^* \circ j^* \circ \sik\\ 
  &=i^* \circ \big( \theta^k(\Tan_j) \circ \sik  + k^{-p-d_j} \cdot \pi^k(\Tan_j)\big)\circ j^*\\
  &=\Big(\theta^k(i^*\Tan_j) \circ \big( \theta^k(\Tan_i) \circ \sik + k^{-p-d_i-d_j}\cdot \pi^k(\Tan_i) \big) + k^{-p-d_j} \cdot \pi^k(i^*\Tan_j)\Big) \circ i^* \circ j^* \\
  &=\theta^k(\Tan_f) \circ \sik \circ f^* + k^{-p-d}\cdot \big( \theta^k(i^*\Tan_j)\circ \pi^k(\Tan_i) + k^{d_i}\cdot \pi^k(i^*\Tan_j)\big) \circ f^*  \\
  &=\theta^k(\Tan_f) \circ \sik \circ f^* + k^{-p-d}\cdot \pi^k(\Tan_f) \circ f^*. \qedhere
\end{align*}
\end{proof}

\section{The first Steenrod square}
In this section, we specialize to the case $k=-1$.\\

Let $f \colon L \to X$ be a line bundle, and $s \colon X \hookrightarrow L$ its zero section. The map $f^* \colon \CK_{p-1,q+1}(X) \to \CK_{p,q}(L)$ is an isomorphism, with inverse $s^*$. The first Chern class of $L$ with values in $\CK$-groups is defined by the formula
\begin{equation}
\label{eq:firstchern}
c_1(L)=s^* \circ s_* \colon \CK_{p,q}(X) \to \CK_{p-1,q+1}(X).
\end{equation}

The $\CK$-groups are naturally endowed with an action of $\Kzl$. The first Chern of $L$ with values in $\Kzl$ (the multiplication by $1-[L^\vee]$) thus provides another way of defining Chern classes of line bundles with values in $\CK$-groups. The next lemma states the expected compatibility between these two constructions. It was proven in \cite[Proposition~6.1]{Cai}, up to composition with $\beta^\infty$ (defined in \ref{sect:by}).

\begin{lemma}
\label{lemm:firstchern}
Let $L$ be a line bundle over a variety $X$. Then $\beta \circ c_1(L)\colon \CK_{p,q}(X) \to \CK_{p,q}(X)$ (see \eqref{eq:firstchern}) coincides with the action of the element $1-[L^\vee] \in \Kzl(X)$. 
\end{lemma}
\begin{proof}
Let $f \colon M \to Y$ be a line bundle, and $s \colon Y \hookrightarrow M$ its zero section. Let $\Mc$ be the $\Oc_Y$-module of sections of the line bundle $M$. The exact sequence
\[
0 \to f^*\Mc^\vee \to \Oc_M \to s_*\Oc_Y \to 0
\]
gives the formula $s_*[\Oc_Y]=1-f^*[M^\vee]$ in $\Kzl(M)$, where $s_* \colon \Kzl(Y) \to \Kzl(M)$ is the push-forward along the perfect morphism $s$ \cite[Example~15.1.8]{Ful-In-98}. Let $x \in \Kg{Y}{p+q}$. The projection formula of \cite[\S7, Proposition~2.10]{Qui-72} gives
\begin{align*}
s^*\circ s_*(x)&=s^*\circ s_*\circ s^*\circ f^*(x)\\
&=s^*(s_*[\Oc_Y] \cdot f^*(x))\\
&=s^*((1-f^*[M^\vee])\cdot f^*(x))\\
&=(1-[M^\vee])\cdot x.
\end{align*}
Using Proposition~\ref{prop:properflataction}, we obtain for $y \in \CK_{p,q}(Y)$, 
\begin{equation}
\label{eq:bycone}
\by \circ c_1(M)(y)=\by ((1-[M^\vee])\cdot y).
\end{equation}

One defines a natural transformation $\eta\colon \CK_{p,q} \to \CK_{p,q}$ of functors $\Var/X \to \Ab$ by setting, for $Y \in \Var/X$, and $y \in \CK_{p,q}(Y)$,
\[
\eta^Y(y)= \beta \circ c_1(L\times_X Y)(y)- (1-[L^\vee\times_X Y])\cdot y.
\]
We conclude using \eqref{eq:bycone} and Proposition~\ref{prop:zero} with $\Omega=X$, and $\mathcal{F}=\id$ (which has relative dimension $\leq 1$).
\end{proof}

For $y\in \Kzh(X)$, we denote by $c_n(y)\colon \CK_{p,q}(X) \to \CK_{p-n,q+n}(X)$ the $n$-th Chern class with values in $\CK$-groups \cite[\S 6.7]{Cai}.
\begin{lemma}
\label{lemm:thetaone}
For all $x \in \Kzh(X)$, we have, as maps $\CK_{p,q}(X) \to \CK_{p,q}(X)$,
\[
\theta^{-1}(x)=\sum_{n\geq 0} (-1)^{\rank(x) -n} \cdot \beta^n \circ c_n(\psi^{-1}(x)).
\]
\end{lemma}
\begin{proof}
Consider, for $x \in \Kzh(X)$, the map $\CK_{p,q}(X) \to \CK_{p,q}(X)$
\[
\nu(x)=(-1)^{\rank(x)}\cdot \sum_{n=0}^\infty (-1)^{-n} \cdot \beta^n \circ c_n(\psi^{-1}(x)).
\]
Both $\nu$ and $\theta^{-1}$ send sums of elements to composition of morphisms, and are characteristic classes in the sense of Lemma~\ref{lemm:characteristic}. Since the pull-back along a projective bundle is injective by \cite[Theorem~6.3]{Cai}, and $\nu(L)=\theta^{-1}(L)$ when $L$ is a line bundle by Lemma~\ref{lemm:firstchern}, we can conclude with the splitting principle.
\end{proof}

\begin{lemma}
\label{lemm:pione}
For all $x \in \Kzh(X)$, we have, as maps $\CK_{p,q}(X) \to \CK_{p-1,q+1}(X)$,
\[
\pi^{-1}(x)=\sum_{n\geq 1} (-1)^{\rank(x) -n} \cdot \beta^{n-1} \circ c_n(\psi^{-1}(x)).
\]
\end{lemma}
\begin{proof}
We have, using Lemma~\ref{lemm:thetaone},
\begin{align*}
i\circ \beta \circ \pi^{-1}(x) \circ j &= i \circ j\circ (\theta^{-1}(x)-(-1)^{\rank(x)}\id)\\ 
  &= i \circ (\theta^{-1}(x)-(-1)^{\rank(x)}\id) \circ j\\
  &=i \circ \Big( \beta \circ \sum_{n\geq 1} (-1)^{\rank(x) -n} \cdot \beta^{n-1} \circ c_n(\psi^{-1}(x)) \Big)\circ j.
\end{align*}
Therefore the image under $\by$ of the claimed equality of natural transformations $\CK_{p,q} \to \CK_{p-1,q+1}$ of functors $\Var/X \to \Ab$ holds true. We conclude using Proposition~\ref{prop:zero} with $\Omega=X$, and $\mathcal{F}=\id$ (which has relative dimension $\leq 0$).
\end{proof}

Recall from \ref{par:connective} that we have an exact sequence
\begin{equation}
\label{seq}
\CK_{p-1,1-p} \xrightarrow{\beta} \CK_{p,-p} \to \CH_p \to 0.
\end{equation}
We denote by $\Ch=\Z/2\otimes\CH \colon \Var \to \Ab$ the Chow group functor modulo two.

\begin{theorem}
The operation $\siu$ induces a natural transformation
\[
\Squ \colon \Ch_{p} \to \Ch_{p-1}.
\]
\end{theorem}
\begin{proof}
The operation $\siu\colon\CK_{p,-p}\to \CK_{p-1,1-p}$ commutes with $\beta$ modulo two by Lemma~\ref{lemm:commutes}, hence induces an endomorphism of $\Z/2 \otimes \coker\beta$. This is precisely the Chow group modulo two because of the sequence \eqref{seq}.
\end{proof}

Using this theorem, we see that the operation $\siu$ can be thought of as a lifting of the first Steenrod square to connective $K$-theory with integral coefficients.\\

Cartan formula below follows from Proposition~\ref{prop:extprod}.
\begin{proposition}
\label{prop:Cartan}
Let $X$ and $Y$ be varieties. If $x \in \Ch_p(X)$ and $y \in \Ch_r(Y)$, we have in $\Ch_{p+r-1}(X \times Y)$
\[
\Squ(x \times y)=\Squ(x) \times y +x \times \Squ(y).
\]
\end{proposition}

\begin{proposition}
\label{prop:pullback}
Let $f \colon Y \to X$ be local complete intersection morphism, with virtual tangent bundle $\Tan_f \in \Kzl(Y)$. Then, as maps $\Ch_p(X) \to \Ch_{p+\rank(\Tan_f)-1}(Y)$,
\[
f^* \circ \Squ=\Squ \circ f^* + c_1(\Tan_f) \circ f^*.
\]
\end{proposition}
\begin{proof}
We apply Proposition~\ref{prop:lci}, \eqref{item:lci:sigma}, Lemmas~\ref{lemm:thetaone} and \ref{lemm:pione}, and compute modulo $2$ and $\beta$,
\begin{align*} 
 f^* \circ \siu &= \theta^{-1}(\Tan_f) \circ \siu \circ f^* + \pi^{-1}(\Tan_f) \circ f^*\\ 
  &= \siu \circ f^* + c_1(\psi^{-1}(\Tan_f)) \circ f^*.
\end{align*}
This gives the result, in view of Lemma~\ref{lemm:conepsi} below.
\end{proof}
Thus if $X$ is a regular variety with structural morphism $x \colon X \to \point$, we have
\begin{equation}
\label{eq:lci}
\Squ[X]=\Squ \circ x^*[\point]=c_1(\Tan_X) \circ x^*[\point] + x^*\circ \Squ[\point]=c_1(\Tan_X).
\end{equation}

\begin{lemma}
\label{lemm:conepsi}
The first Chern class $c_1$ with values in Chow groups modulo two coincides with $c_1 \circ \psi^{-1}$.
\end{lemma}
\begin{proof}
This follows from the splitting principle.
\end{proof}

\begin{remark}[Cohomological operation]
\label{rem:coh}
When $X$ is a smooth variety, and $x \in \Ch(X)$, we set $\Squu(x)=\Squ(x) + c_1(\Tan_X)\cdot x \in \Ch(X)$. Then
\begin{enumerate}[(a)]
\item \label{it:a} $\Squu(a\cdot b)=\Squu(a)\cdot b + a \cdot \Squu(b).$

\item \label{it:b} If $x \in \Ch^1(X)$, then $\Squu(x)=x^2$.

\item \label{it:c} Let $f \colon Y \to X$ be a morphism of smooth varieties. Then $f$ is a local complete intersection morphism \cite[Corollary~104.4]{EKM}, and $f^* \circ \Squu=\Squu \circ f^*$.
\end{enumerate}

Indeed \eqref{it:c} follows from  Proposition~\ref{prop:pullback}, and \eqref{it:a} is a consequence of \eqref{it:c} and Proposition~\ref{prop:Cartan}. While proving \eqref{it:b}, we may assume that $x=[Z]$ for some integral closed subvariety $i\colon Z\hookrightarrow X$ of codimension one. Since $X$ is locally factorial, $i$ is a regular closed embedding. Using \eqref{eq:lci} and Proposition~\ref{prop:pullback}, we get \eqref{it:b} since 
\begin{align*}
\Squ(x)&=\Squ \circ i_*\circ i^*[X]=i_*\circ \Squ \circ i^*[X]=i_* (i^* \circ \Squ[X]+ c_1(i^*\Oc_X(Z)))\\
&= c_1(\Oc_X(Z))(c_1(\Tan_X) + c_1(\Oc_X(Z)))\\
&= c_1(\Tan_X)\cdot x +x^2.
\end{align*}
\end{remark}

The last property that we discuss is the Adem relation. We are not able to prove the relation $\Squ \circ \Squ=0$, but we shall prove a slightly weaker statement.

Let $X$ be a variety. Consider the topological filtration of $\Kzl'(X)$, whose $p$-th term is given by the image of $\K_0(\Mc_p(X))$, and denote by $\gr_p \K'_0(X)=E^\infty_{p,-p}$ the associated graded group. This defines a functor $\gr_p \K'_0 \colon \Var \to \Ab$.

\begin{proposition}
\label{prop:descends}
The transformation $\Squ$ descends to a natural transformation
\[
\mathfrak{S}_1 \colon \Z/2 \otimes \gr_p \K'_0 \to \Z/2 \otimes \gr_{p-1} \K'_0.
\]
\end{proposition}
\begin{proof}
Let $\omega \colon \CK_{p,-p} \to \Ch_p$ be the natural epimorphism. We have
\[
\by \circ \siu=(\psi_{-1} - (-1)^{-p}\id) \circ \by,
\]
hence $\siu(\ker \by) \subset \ker \by$. It follows that 
\[
\Squ \circ \omega(\ker \by) =\omega \circ \siu(\ker \by) \subset \omega(\ker \by).
\]
The claim is then a consequence of the exact sequence
\[
\ker\by \xrightarrow{\omega} \Ch_p \to \Z/2\otimes \gr_p \K'_0 \to 0. \qedhere
\]
\end{proof}

\begin{lemma}
\label{lemm:last}
We have, as maps $\CK_{p,q} \to \CK_{p-2,q+2}$ and $\CK_{p,q} \to \CK_{p-1,q+1}$,
\[
2 \cdot\siu \circ \siu=0 \quad \text{ and } \quad \siu \circ \siu \circ \beta=\beta \circ \siu \circ \siu=0.
\]
\end{lemma}
\begin{proof}
Recall from Proposition~\ref{prop:adams}, \eqref{adams:compos} that $\psi_{-1}^{(p,q)} \circ \psi^{(p,q)}_{-1}=\id$. We compute, as maps $\K_{p+q}(\Mc_p(-)) \to \K_{p+q}(\Mc_p(-))$, for $p+q \leq 1$,
\begin{align*} 
    i \circ \siu \circ \beta \circ \siu \circ j &= i \circ \siu \circ j \circ i \circ \siu \circ j \\ 
    &=(\psi^{(p,q)}_{-1} - (-1)^q\id)\circ (\psi^{(p,q)}_{-1} - (-1)^q\id)\\
    &=-2(-1)^q \cdot (\psi^{(p,q)}_{-1} - (-1)^q\id)\\
    &=i \circ (-2(-1)^q\cdot\siu) \circ j.
\end{align*}
It follows that, as maps $\CK_{p,q} \to \CK_{p-1,q+1}$, 
\begin{equation}
\label{eq:sigmaone}
\siu \circ \beta \circ \siu=-2(-1)^q\cdot\siu.
\end{equation}
From Lemma~\ref{lemm:commutes}, we obtain, as maps $\CK_{p,q} \to \CK_{p-1,q+1}$,
\[
\beta \circ \siu \circ \siu=\siu \circ \beta \circ \siu + 2(-1)^q\cdot\siu=0,
\]
and the first relation then follows by composing on the right by $\siu$ the relation~\eqref{eq:sigmaone}. The remaining relation can be proved in a similar fashion.
\end{proof}

From the first relation, we see that the image of $\Squ \circ \Squ$ lies inside the image of the two-torsion subgroup under the map $\CH \to \Ch$. More precisely:

\begin{proposition}
The following composite is zero ($E^n_{p,q}$ being Quillen spectral sequence, see \ref{sect:QSS}):
\[
\Ch_p \xrightarrow{\Squ} \Ch_{p-1} \xrightarrow{\Squ} \Ch_{p-2} \to  \Z/2 \otimes E^3_{p-2,2-p}.
\]
\end{proposition}
\begin{proof}
The kernel of the map $\CH_{p-2} \to  E^3_{p-2,2-p}$ is the image of $\ker \beta \subset \CK_{p-2,2-p}$. The statement follows from the relation $\beta \circ \siu \circ \siu=0$ of Lemma~\ref{lemm:last} above.
\end{proof}

\begin{corollary}
The following composite is zero:
\[
\Z/2 \otimes \gr_p \K'_0 \xrightarrow{\mathfrak{S}_1}  \Z/2 \otimes \gr_{p-1}  \K'_0  \xrightarrow{\mathfrak{S}_1} \Z/2 \otimes \gr_{p-2} \K'_0.
\]
\end{corollary}

\section{Torsion in the first Chow group of quadrics}
Let $X, Y$ be two varieties. A \emph{correspondence} (resp.\ \emph{modulo two correspondence}) is an element of $r \in \CH_{\dim X}(X\times Y)$ (resp.\ $r\in\Ch_{\dim X}(X\times Y))$. Let $q\colon X \times Y \to X$ and $p\colon X \times Y \to Y$ be the two projections. Assume that $Y$ is projective and $X$ integral. The \emph{multiplicity} of $r$ is the element of $m$ of $\Z$ (resp.\ $\Z/2$) such that $q_*(r)=m\cdot[X]$.

If $X$ and $Y$ are smooth projective, we define a push-forward $r_*$ by the formula
\[
r_*(x)=p_*(q^*(x) \cdot r).
\]

Let $X$ be a projective variety without zero-cycles of odd degree. We define a homomorphism, by assigning to a cycle the half of its degree, taken modulo two,
\[
\deg/2 \colon \Ch_0(X) \to \Z/2.
\]

\begin{lemma}
\label{lemm:corr}
Let $X$ be a smooth, connected, projective variety without zero-cycles of odd degree, $r \in \Ch(X \times X)$ a modulo two correspondence of multiplicity $m$, and $x \in \Ch_1(X)$. Then we have, in $\Z/2$,
\[
(\deg/2)\circ \Squ \circ r_*(x)=m \cdot (\deg/2) \circ \Squ(x).
\]
\end{lemma}
\begin{proof}
Let $q,p \colon X \times X\to X$ be the two projections. We have 
\begin{align*} 
(\deg/2)\circ \Squ \circ r_*(x)  &= (\deg/2)\circ \Squ \circ p_*( q^*(x)\cdot r)\\ 
  &= (\deg/2)\circ p_* \circ \Squ ( q^*(x)\cdot r)\\
  &= (\deg/2)\circ \Squ ( q^*(x)\cdot r)\\
  &= (\deg/2)\circ q_* \circ \Squ ( q^*(x)\cdot r)\\
  &= (\deg/2)\circ \Squ \circ q_*(q^*(x)\cdot r)\\
  &= (\deg/2)\circ \Squ ( x \cdot q_*(r))\\
  &= m \cdot (\deg/2)\circ \Squ(x).\qedhere
   \end{align*}
\end{proof}

\begin{proposition}
\label{prop:nonzero}
Let $X$ be an anisotropic smooth projective quadric of dimension $d \geq 1$. Let $r \in \Ch(X \times X)$ be a modulo two correspondence of multiplicity one. 

Then $r_*\Ch_1(X)\neq0$.
\end{proposition}
\begin{proof}
Let $h\in\Ch^1(X)$ be the hyperplane class. By \cite[Lemma~78.1]{EKM},
\begin{equation}
\label{eq:c1tan}
c_1(\Tan_X)=(d+2) \cdot h.
\end{equation}
For any integer $n\geq 0$, we have by Remark~\ref{rem:coh},~\eqref{it:a} and \eqref{it:b}
\[
\Squu(h^n)=\Squu(h \cdot h^{n-1})=\Squu(h)\cdot h^{n-1} + h \cdot \Squu(h^{n-1})=h^{n+1} + h \cdot \Squu(h^{n-1}).
\]
By induction on $n$, this gives
\begin{equation}
\label{eq:squuhn}
\Squu(h^n)=n \cdot h^{n+1}. 
\end{equation}
Using Lemma~\ref{lemm:corr}, letting $n=d-1$ in \eqref{eq:squuhn}, and using \eqref{eq:c1tan}, we obtain
\begin{align*}
(\deg/2) \circ \Squ \circ r_*(h^{d-1})&=(\deg/2) \circ \Squ(h^{d-1})\\
&=(\deg/2)(\Squu(h^{d-1})+c_1(\Tan_X)\cdot h^{d-1})\\
&=(\deg/2)( (d-1)\cdot h^d + (d+2)\cdot h^d)\\
&=(\deg/2)(h^d)\\
&=1 \mod 2.
\end{align*}
In particular $r_*(h^{d-1}) \in r_*\Ch_1(X)$ is non-zero.
\end{proof}

\begin{proposition}
\label{prop:torsion}
Let $X$ be an anisotropic smooth projective quadric of dimension $d \geq 3$. Assume that there is a modulo two correspondence $r \in \Ch(X \times X)$ of multiplicity one such that $(r_L)_*\Ch_1(X_L)=0$, for some algebraic closure $L$ of the base field. Then there is a non-zero cycle $\xi \in \CH_1(X)$ such that $2\cdot \xi=0$.
\end{proposition}
\begin{proof}
Since $d \geq 3$ and $X$ is anisotropic, we have by \cite[Corollary~72.2]{EKM},
\begin{equation}
\label{eq:2}
2 \cdot \CH_1(X_L)=\CH_1(X)_L.
\end{equation}
(Indeed by \cite[\S68]{EKM} we have $\CH_1(X_L)=\Z \cdot l_1$, with $2\cdot l_1=h^{d-1} \in \CH_1(X)_L$. This proves one inclusion. To prove the other one, note that $l_1 \not \in \CH_1(X)_L$; for otherwise $h\cdot l_1$ would be a zero-cycle of degree one in $\CH_1(X)_L$, contradicting the anisotropy of $X$.) By Proposition~\ref{prop:nonzero}, we can find a cycle $\alpha \in \CH_1(X)$ such that $r_*(\alpha \mod 2)\neq 0$. If $R$ is an integral lifting of $r_L$, then $R_*(\alpha_L)$ belongs to 
\begin{align*}
R_*(\CH_1(X)_L)&\subset 2 \cdot R_* \CH_1(X_L)&&\text{by \eqref{eq:2}}\\
&\subset 4 \cdot \CH_1(X_L)&&\text{by the hypothesis on $r$}\\
&\subset 2 \cdot \CH_1(X)_L&&\text{by \eqref{eq:2}}.
\end{align*}
It follows that we can find $\xi \in \CH_1(X)$ such that $\xi \mod 2=r_*(\alpha\mod 2)$ (in particular $\xi \neq 0$) and $\xi_L=0$. We conclude using Lemma~\ref{lemm:first} below, for $i=1$.
\end{proof}

\begin{lemma}
\label{lemm:first}
Let $X$ be an anisotropic smooth projective quadric, $\mathfrak{i}_1(X)$ its first Witt index \cite[Notation~25.5]{EKM}. Assume that $\mathfrak{i}_1(X)\geq i$. Let $L$ be an extension of the base field. Then
\[
2 \cdot \ker(\CH_i(X) \to \CH_i(X_L))=0. 
\]
\end{lemma}
\begin{proof}
We can assume that $L$ is algebraically closed. Let $K$ be a extension of the base field contained in $L$, such that $X_K$ is isotropic. Let $X' \subset X_K$ be a smooth subquadric (over $K$) of codimension $2i$, given by a quadratic form Witt-equivalent to a quadratic form giving $X_K$. Applying $i$ times  \cite[Proposition~70.1]{EKM}, we obtain an isomorphism $\CH_0(X') \simeq \CH_i(X_K)$ which is compatible with extension of scalars (in case $2i=\dim X$, we have $\CH_i(X_K)=\CH_0(X')\oplus H$, the group $H$ being the free abelian group generated by the $i$-th power of the hyperplane class). Since the map $\CH_0(X') \to \CH_0(X'_L)$ is injective by \cite[Corollary~71.4]{EKM} (and so is $H\to H_L$ in case $2i=\dim X$), so is $\CH_i(X_K) \to \CH_i(X_L)$.

Now we can take for $K$ a separable quadratic extension of the base field, and a transfer argument yields the result.
\end{proof}

In the language of \cite{V}, the conditions of Proposition~\ref{prop:torsion} mean that the Tate motives $\Z$ and $\Z(1)[2]$ are not connected in $\Lambda(X)$. Examples of anisotropic smooth projective quadrics $X$ satisfying this condition include: 
\begin{enumerate}[(i)]
\item Quadrics whose first Witt index $\mathfrak{i}_1$ is greater than one (take $r=(1 \times h^{\mathfrak{i}_1-1})\cdot \pi$, where $\pi$ is the $1$-primordial cycle, see \cite[Theorem~73.26]{EKM}; in characteristic not two, this also follows from \cite[Corollary~3.10, Lemma~4.2]{V}).

\item \label{it:rost} Quadrics with an integral Rost projector $\rho$ \cite[Section~2]{Rostprojectors}: the main (conjecturally only) examples being minimal Pfister neighbours.\\
\end{enumerate}

We now assume that the base field has characteristic different from two. Then the situation~\eqref{it:rost} has been extensively studied using the Steenrod squares. The integer $\dim X +1$ is known to be a power of two (\cite[Theorem~6.1]{IV-Max}, or \cite[Theorem~5.1]{Rostprojectors}), and we have, by \cite[Corollary~8.2]{Rostprojectors},
\begin{equation}
\label{Chxrho}
\rho_*\CH_p(X)=\left\{ \begin{array}{rl}
		  0 &\mbox{ if $p+1$ is not a power of two}, \\
		  \Z &\mbox{ if $p=0$ or $\dim X$},\\
		  \Z/2 &\mbox{ if $p+1$ is a power of two and $p\neq 0, \dim X$}.
       		\end{array}
	\right.
\end{equation}

According to \cite[Section~6]{Kar-minimal}, the image of $\rho_*\CH_p(X)$ in $\gr_p \K'_0(X)$ is 
\begin{equation}
\label{Kxrho}
\left\{ \begin{array}{rl}
		  0 &\mbox{ if $1<p<\dim X$}, \\
		  \Z &\mbox{ if $p=0$ or $\dim X$},\\
		  \Z/2 &\mbox{ if $p=1$}.
       		\end{array}
	\right.
\end{equation}
The base field has characteristic not two in \cite[Section~6]{Kar-minimal}, but this assumption could probably be dropped (since no Steenrod square is used there) --- this would provide another approach to the proof of Proposition~\ref{prop:torsion} for the special case \eqref{it:rost}.

Using \eqref{Kxrho}, we see that the element $\xi$ of Proposition~\ref{prop:torsion} is the only torsion cycle in $\rho_*\CH(X)$ surviving in $\gr \K'_0(X)$. If $\mu$ is the quotient map killing the image modulo two of torsion cycles, then $\mu \circ \Squ(\xi \mod 2)\neq 0$ (this is the content of the proof of Proposition~\ref{prop:nonzero}). Therefore the Rost motive $(X,\rho)$ provides a non-trivial example where the operation $\Squ$ is sufficient to detect cycles vanishing in $\gr \K'_0$ among the torsion cycles.

\begin{remark}
\label{rem:other}
In characteristic different from two, Proposition~\ref{prop:descends} does not generalize to all the Steenrod squares. Indeed it follows from \eqref{Kxrho}, and \cite[Corollary~4.9]{Rostprojectors}, that the morphisms $\mu \circ \Square^{2^i-1}$ do not factor through $\Z/2 \otimes \gr\K'_0(-)$ when $i>1$ (here $\Square^n$ is the $n$-th cohomological Steenrod square).
\end{remark}
\bibliographystyle{alpha} 

\end{document}